\documentclass[reqno]{amsart}
\usepackage{amsmath,amsthm,amssymb,amscd}
\usepackage[shortlabels]{enumitem}

\newtheorem{theorem}{Theorem}[section]
\newtheorem{lemma}[theorem]{Lemma}
\newtheorem{proposition}[theorem]{Proposition}

\newtheorem{corollary}[theorem]{Corollary}
\theoremstyle{definition}
\newtheorem{definition}[theorem]{Definition}

\theoremstyle{remark}
\newtheorem{remark}[theorem]{Remark}
\newtheorem{conj}[theorem]{Conjecture}
\numberwithin{equation}{section}
\allowdisplaybreaks
\begin{document}

%
%
%
%
%
%
%
%
%

\title[New parity results for $21$-regular partitions]
 {Certain Diophantine equations and new parity results for $21$-regular partitions}

\author{Ajit Singh}
\address{Department of Mathematics, Indian Institute of Technology Guwahati, Assam, India, PIN- 781039}
\email{ajit18@iitg.ac.in}

\author{Gurinder Singh}
\address{Department of Mathematics, Indian Institute of Technology Guwahati, Assam, India, PIN- 781039}
\email{gurinder.singh@iitg.ac.in}

\author{Rupam Barman}
\address{Department of Mathematics, Indian Institute of Technology Guwahati, Assam, India, PIN- 781039}
\email{rupam@iitg.ac.in}

\date{January 27, 2023}


\subjclass{11D09, 11D45, 11P83}

\keywords{Diophantine equations; Dirichlet density; modular forms; regular partitions}

\dedicatory{}

\begin{abstract} For a positive integer $t\geq 2$, let $b_{t}(n)$ denote the number of $t$-regular partitions of a nonnegative integer $n$. In a recent paper, Keith and Zanello investigated the parity of $b_{t}(n)$ when $t\leq 28$. They discovered new infinite families of Ramanujan type congruences modulo 2 for $b_{21}(n)$ involving every prime $p$ with $p\equiv 13, 17, 19, 23 \pmod{24}$. In this paper,
we investigate the parity of $b_{21}(n)$ involving the primes $p$ with $p\equiv 1, 5, 7, 11 \pmod{24}$. We 
prove new infinite families of Ramanujan type congruences modulo 2 for $b_{21}(n)$ involving the odd primes $p$ for which the Diophantine equation $8x^2+27y^2=jp$ has primitive solutions for some $j\in\left\lbrace1,4,8\right\rbrace$, and we also prove that the Dirichlet density of such primes is equal to $1/6$. Recently, Yao provided new infinite families of congruences modulo $2$ for $b_{3}(n)$ and those congruences involve every prime $p\geq 5$ based on Newman's results. Following a similar approach, we prove new infinite families of congruences modulo $2$ for $b_{21}(n)$, and these congruences imply that $b_{21}(n)$ is odd infinitely often.
\end{abstract}

\maketitle
\section{Introduction and statement of results} 
 A partition of a positive integer $n$ is a finite sequence of non-increasing positive integers $(\lambda_1, \lambda_2, \ldots, \lambda_k)$ such that $\lambda_1+\lambda_2+\cdots +\lambda_k=n$. 
 Let $t\geq 2$ be a fixed positive integer. A $t$-regular partition of a positive integer $n$ is a partition of $n$ such that none of its part is divisible by $t$. For example, $(5, 2, 2, 1)$ is a $3$-regular partition of $10$. Let $b_{t}(n)$ be the number of $t$-regular partitions of $n$. The generating function for $b_{t}(n)$ is given by 
\begin{align}\label{gen_fun}
\sum_{n=0}^{\infty}b_{t}(n)q^n=\frac{f_{t}}{f_1},
\end{align}
where $f_k:=(q^k; q^k)_{\infty}=\prod_{j=1}^{\infty}(1-q^{jk})$ and $k$ is a positive integer. Also, $b_t(0)$ is defined to be $1$.
 \par In literature, many authors have studied the arithmetic properties of $b_{t}(n)$ for certain values of $t$, see for example \cite{Merca, baruahdas, calkin, Carlson, cui_Gu_2013, cui_Gu_2015, Dandurand, Furcy, gordon1997, hirschhorn_sellers, Keith2014, Keith2021, Lovejoy, Penniston_2002, Penniston_2008, Webb, Xia_2015, xia_yao,Yao2022, zhao_jin_yao}. In a recent paper \cite{Keith2021}, Keith and Zanello studied the parity of  $b_t(n)$ for certain values of $t\leq 28$, and made several conjectures on $b_t(n)$. 
In \cite{SB2}, the first and the third author have proved two conjectures of Keith and Zanello on $b_3(n)$ and $b_{25}(n)$ respectively. In \cite[Theorem 13]{Keith2021}, Keith and Zanello proved infinite families of Ramanujan type congruences satisfied by $b_{21}(n)$ which involves every prime $p$ satisfying $p\equiv 13,17,19,23 \pmod{24}$. To be specific, if $p\equiv 13,17,19,23 \pmod{24}$ is prime, then they proved that
\begin{align}\label{KeithThm}
b_{21}(4(p^2n+kp-11\cdot 24^{-1})+1)\equiv 0\pmod{2}
\end{align}
for all $1\leq k<p$, where $24^{-1}$ is taken modulo $p^2$.
\par The aim of this paper is to investigate the parity of $b_{21}(n)$. Our first main result proves infinite families of Ramanujan type congruences for the remaining classes of primes modulo $24$, namely, $p\equiv 1,5,7,11 \pmod{24}$. A key ingredient in our proof is to find the integral solutions of certain Diophantine equations, which has also been employed very recently by Ballantine, Merca, and Radu in \cite{Merca} to study the parity of $b_{3}(n)$. To state our main results, we first introduce some notation. Given the Diophantine equation $ax^2+by^2 = n$, by a primitive solution we mean a solution $(x, y)\in \mathbb{Z}^2$ satisfying $\gcd(x, y) = 1$. Let $\mathcal{Q}$ denote the set of odd primes $p$ such that the Diophantine equation $8x^2+27y^2=jp$ has primitive solutions for some $j\in \{1, 4, 8\}$.
\begin{theorem}\label{thm1}
For every $p\in\mathcal{Q}$ and $n\geq0$,  we have
\begin{align*}
b_{21}(4(p^2n+\beta p-11\cdot24^{-1}_p)+1)\equiv 0 \pmod{2}
\end{align*}
for all $\beta$ with $0\leq\beta<p$ and $\displaystyle\beta\neq\frac{11(p^2+24\cdot24^{-1}_p-1)}{24p}$, where $24^{-1}_p$ is the inverse of $24$ taken modulo $p$ such that $1\leq-24^{-1}_p\leq p-1$.
\end{theorem}
In section 2, we show that, if $p\in \mathcal{Q}$, then $p\equiv 1,5,7,11\pmod{24}$, and $\mathcal{Q}$ contains infinitely many primes from each of these classes of primes modulo $24$.
To study the parity of $b_{3}(n)$, Ballantine, Merca, and Radu \cite{Merca} considered the set $\mathcal{P}$ which contains the primes $p$ such that the Diophantine equation $x^2+24\cdot 9y^2=jp$ has primitive solutions for some $j\in \{1, 4, 8\}$. We also show that $\mathcal{Q}=\mathcal{P}$.
\par We have $\mathcal{Q}=\{29, 59, 79, 103, 223, 227, 241, \ldots \}$. Putting $p=29$ in Theorem \ref{thm1}, we have
\begin{align}\label{eqn-thm1-29}
b_{21}(4(29^2n+29\beta+66)+1)\equiv 0\pmod 2,
\end{align}
for all $0\leq\beta<29$, $\beta\neq11$.
In \cite[Theorem 1.4]{SB1}, the first and the third author proved the congruence \eqref{eqn-thm1-29} using a technique developed by Radu in \cite{radu1}. In section 2, we also prove Theorem \ref{thm1} when $p=59, 79$ using Radu's technique. We note that proving Theorem \ref{thm1} for larger values of $p$ using  Radu's technique is not computationally feasible.
\par 
In a recent paper \cite{Yao2022}, Yao provided new infinite families of congruences modulo $2$ for $b_{3}(n)$ and those congruences involve every prime $p\geq 5$ based on Newman's results. In our next theorem, we prove similar congruences for $b_{21}(n)$. These congruences involve every prime $p\geq 5$ except $p=11$. In addition, we show that if $p\geq5$ with $p\not=11$ is a prime, then there exists an integer $\mu( p)\in\{4,6\}$ such that for all $k\geq 0$, $b_{21}(\frac{11p^{\mu(p)k}-5}{6})$ is odd. More precisely, we have the following theorem.
\begin{theorem}\label{Newmann_b21}
	Let $p\geq5$ with $p\not=11$ be a prime.
	\begin{enumerate}
		\item[\emph{(i)}] If $b_{21}\left(\frac{11p^2-5}{6}\right)\equiv1\pmod{2}$, then for all $n,k\geq0$,
		\begin{align}\label{Yao_1}
		b_{21}\left( 4p^{4k+4}n+4p^{4k+3}\beta+\frac{11p^{4k+4}-5}{6}\right) \equiv0\pmod{2},
		\end{align}
		where $1\leq \beta<p$ and for $k\geq0$,
		\begin{align}\label{Yao_1a}
		b_{21}\left(\frac{11p^{4k}-5}{6}\right)\equiv 1\pmod{2}. 
		\end{align}
		\item[\emph{(ii)}] If $b_{21}\left(\frac{11p^2-5}{6}\right)\equiv0\pmod{2}$, then for all $n,k\geq0$ with $p\nmid(24n+11)$,
		\begin{align}\label{Yao_2}
		b_{21}\left( 4p^{6k+2}n+\frac{11p^{6k+2}-5}{6}\right) \equiv0\pmod{2}
		\end{align}
		and for $k\geq0$,
		\begin{align}\label{Yao_2a}
		b_{21}\left(\frac{11p^{6k}-5}{6}\right)\equiv 1\pmod{2}. 
		\end{align}
	\end{enumerate}	
\end{theorem}   
\section{Proof of Theorem \ref{thm1}}
We first recall Euler's Pentagonal Number Theorem \cite{Andrews},
\begin{align}\label{thm1.2New}
	f_1=\sum_{m\in\mathbb{Z}}(-1)^mq^{\frac{m}{2}(3m-1)}.
\end{align}
From Lemma 18 of \cite{Keith2021}, we have
\begin{align}\label{thm1.2}
	\frac{f^3_3}{f_1}\equiv\sum_{n\in\mathbb{Z}}q^{n(3n-2)}\pmod{2}.
\end{align}
We next recall the following identity from \cite{Keith2021}:
\begin{align}\label{thm1.1}
	\sum_{n=0}^{\infty}b_{21}(4n+1)q^n\equiv \frac{f_3^4}{f_1}\pmod 2.
\end{align}
Combining \eqref{thm1.2New}, \eqref{thm1.2}, and \eqref{thm1.1}, we have
\begin{align*}
	\sum_{n=0}^{\infty}b_{21}(4n+1)q^n&\equiv\left(\frac{f^3_3}{f_1}\right)f_3\\
	&\equiv\sum_{n\in\mathbb{Z}}q^{n(3n-2)} \sum_{m\in\mathbb{Z}}q^{\frac{3m}{2}(3m-1)}\pmod{2}.
\end{align*}
We define 
\begin{align*}
	\sum_{k=0}^{\infty}a(k)q^k:=\sum_{n\in\mathbb{Z}}q^{n(3n-2)} \sum_{m\in\mathbb{Z}}q^{\frac{3m}{2}(3m-1)}.
\end{align*}
Then, for all $k\geq0$, we have
\begin{align}\label{relation_b21_and_a}
	b_{21}(4k+1)\equiv a(k)\pmod{2},
\end{align}
where
\begin{align*}
	a(k)=\left|\left\lbrace(n,m)\in\mathbb{Z}^2:\ n(3n-2)+\frac{3m}{2}(3m-1)=k\right\rbrace\right|
\end{align*}
which we can further rewrite as
\begin{align*}
	a(k)=\left|\left\lbrace(x,y)\in\mathbb{N}^2:\ 8x^2+3y^2=24k+11,\ 3\nmid y\right\rbrace\right|.
\end{align*}
To prove Theorem \ref{thm1}, it is enough to prove that $a(p^2n+\beta p-11\cdot24^{-1}_p)\equiv 0 \pmod{2}$, for those $p$ and $\beta$ given in the statement of the theorem. To prove this claim, we require the following lemmas.
\begin{lemma}\label{lemma_1}
Let $p\geq 29$ be a prime and $0\leq\beta<p$, $\displaystyle\beta\neq\frac{11(p^2+24\cdot24^{-1}_p-1)}{24p}$, and $t=p^2n+\beta p-11\cdot24^{-1}_p$. Then $24t+11=pm$, for some positive integer $m$ such that $p\nmid m$.
\end{lemma}
\begin{proof}
Since $24^{-1}_p$ is the inverse of 24 modulo $p$, there exists an integer $k$ such that $24\cdot24^{-1}_p=kp+1$. Then $24t+11=24p^2n+24\beta p-11kp$. Clearly, $p\mid24t+11$, and therefore, there exists a positive integer $m$ (since $t$ is positive) such that $24t+11=pm$.
\par Suppose that $p\mid m$. Then $p^2\mid24(p^2n+\beta p-11\cdot24^{-1}_p)+11$. This in turn implies that $p^2\mid24(\beta p-11\cdot24^{-1}_p)+11$. Then, there exists $u\in\mathbb{Z}$ such that
\begin{align}{\label{up^2}}
24(\beta p-11\cdot24^{-1}_p)+11=up^2.	
\end{align}
From the inequalities $0\leq\beta<p$ and $1\leq-24^{-1}_p\leq p-1$, we get
\begin{align*}
	\frac{25\cdot 11}{p^2}\leq \frac{24(\beta p-11\cdot24^{-1}_p)+11}{p^2}<24+\frac{11}{p^2}(24p-23).	
\end{align*}
Since $24p-23<p^2$ for every $p\geq29$, we have $0\leq u<24+11$. Also, \eqref{up^2} and the fact $p^2\equiv1\pmod{24}$ imply that $u\equiv11\pmod{24}$. Therefore, $u=11$. Thus, $24(\beta p-11\cdot24^{-1}_p)+11=11p^2$ and then $\displaystyle\beta=\frac{11(p^2+24\cdot24^{-1}_p-1)}{24p}$, which is a contradiction.
\end{proof}
Now, for $j\in\left\lbrace1,4,8\right\rbrace$, we observe the behavior modulo 24 of a positive integer $m$, coprime to 6, when the equation
\begin{equation}{\label{quad_form_8_27_jm}}
8x^2+27y^2=jm
\end{equation}
has primitive solutions.
\begin{lemma}\label{lemma_j=1}
If, for $j=1$ and a positive integer $m$ coprime to $6$, \eqref{quad_form_8_27_jm} has primitive solutions, then $m\equiv11\pmod{24}$.
\end{lemma}
\begin{proof}
Let $m$ be a positive integer with $\gcd(m,6)=1$ and let $(x_0,y_0)$ be a primitive solution of \eqref{quad_form_8_27_jm} with $j=1$. We have $8x^2_0+27y^2_0=m$. Then, $2\nmid y_0$ and $3\nmid x_0$. Thus, $y^2_0\equiv1,9\pmod{24}$ and $x^2_0\equiv1,4,16\pmod{24}$. This yields that $m\equiv8x^2_0+3y^2_0\equiv11\pmod{24}$.
\end{proof}
\begin{lemma}\label{lemma_j=4}
If, for $j=4$ and a positive integer $m$ coprime to $6$, \eqref{quad_form_8_27_jm} has primitive solutions, then $m\equiv5\pmod{24}$.
\end{lemma}
\begin{proof}
Let $m$ be a positive integer with $\gcd(m,6)=1$ and let $(x_0,y_0)$ be a primitive solution of \eqref{quad_form_8_27_jm} with $j=4$. We have $8x^2_0+27y^2_0=4m$. Clearly, $2\mid y_0$ and therefore, we write $y_0=2w$ for some $w\in\mathbb{Z}$. This gives $2x^2_0+27w^2=m$. Therefore, $2\nmid w$ and $3\nmid x_0$. Then, $w^2\equiv1,9\pmod{24}$ and $x^2_0\equiv1,4,16\pmod{24}$. Since $\gcd(x_0,y_0)=1$ and $y_0$ is even, $x^2_0\equiv1\pmod{24}$. Therefore, $m\equiv2+3w^2\equiv5\pmod{24}$.
\end{proof}
\begin{lemma}\label{lemma_j=8}
If, for $j=8$ and a positive integer $m$ coprime to $6$, \eqref{quad_form_8_27_jm} has primitive solutions, then $m\equiv1,7\pmod{24}$. Let, for a positive integer $n$, $val_p(n)$ denote the exponent of the highest power of the prime $p$ that divides $n$. For $j=8$, if $(x_0,y_0)$ is a primitive solution of \eqref{quad_form_8_27_jm}, then
\begin{enumerate}
\item[\emph{(a)}] if $m\equiv1\pmod{24}$, then $val_2(y_0)\geq3$;
\item[\emph{(b)}] if $m\equiv7\pmod{24}$, then $val_2(y_0)=2$.
\end{enumerate}
\end{lemma}
\begin{proof}
Let $m$ be a positive integer with $\gcd(m,6)=1$ and let $(x_0,y_0)$ be a primitive solution of \eqref{quad_form_8_27_jm} with $j=8$. We have $8x^2_0+27y^2_0=8m$. Clearly, $4\mid y_0$ and therefore, we write $y_0=4w$ for some $w\in\mathbb{Z}$. This gives $x^2_0+54w^2=m$. Then, $2,3\nmid x_0$ and therefore, $x^2_0\equiv1\pmod{24}$. Thus, $m\equiv1+6w^2\pmod{24}$.
\par Now, if $2\mid w$ then $m\equiv1\pmod{24}$. If $2\nmid w$ then $w^2\equiv1,9\pmod{24}$ and thus, $m\equiv7\pmod{24}$. Also, if $m\equiv1\pmod{24}$ then $y_0=4w$ and $2\mid w$, thus $val_2(y_0)\geq3$. On the other hand, if $m\equiv7\pmod{24}$ then $y_0=4w$ and $2\nmid w$, thus $val_2(y_0)=2$.
\end{proof}
The above three lemmas imply that if $p\in\mathcal{Q}$, then
\begin{align*}
&j=1\Longrightarrow p\equiv 11\pmod{24};\\
&j=4\Longrightarrow p\equiv 5\pmod{24};\\
&j=8\Longrightarrow p\equiv 1,7\pmod{24}.
\end{align*}
\begin{lemma}{\label{U_p,m_A_m}}
Let $p\in\mathcal{Q}$. Let $m$ be a positive integer such that $p\nmid m$ and $pm\equiv11\pmod{24}$. Let $j\in\left\lbrace1,4,8\right\rbrace$ be such that $8x^2+27y^2=jp$ has primitive solutions. If
\begin{align*}
U_{p,m}&:=\left\lbrace(u,v)\in\mathbb{Z}^2:8u^2+27v^2=pm,~\gcd(u,v)=1\right\rbrace;\\
A_m&:=\left\lbrace(a,b)\in\mathbb{Z}^2:a^2+216b^2=jm,~\gcd(a,b)=1\right\rbrace,
\end{align*}
then $|U_{p,m}|=2|A_m|$.
\end{lemma}
\begin{proof}
Let $p\in\mathcal{Q}$ and $m$ be a positive integer coprime to $p$ such that $pm\equiv11\pmod{24}$. Let $j\in\left\lbrace1,4,8\right\rbrace$ be such that $8x^2+27y^2=jp$ has primitive solutions and let $x_1,~y_1\in\mathbb{Z}$ with $\gcd(x_1,y_1)=1$ satisfying $8x^2_1+27y^2_1=jp$. Define a map $f:U_{p,m}\rightarrow A_m$ by $f((u,v))=(a,b)$, where for $(u,v)\in U_{p,m}$, we define $(a,b)$ as follows. 
\begin{enumerate}[(a)]
\item If $8x_1u-27y_1v\equiv0\pmod{p}$ then $x_1v+y_1u\equiv0\pmod{p}$ and define
\begin{align}\label{definition_a_b_1}
a:=\frac{8x_1u-27y_1v}{p},\ \ \ \  b:=\frac{x_1v+y_1u}{p}.
\end{align}
\item If $8x_1u+27y_1v\equiv0\pmod{p}$ then $x_1v-y_1u\equiv0\pmod{p}$ and define
\begin{align}\label{definition_a_b_2}
a:=\frac{8x_1u+27y_1v}{p},\ \ \ \  b:=\frac{x_1v-y_1u}{p}.
\end{align}
\end{enumerate}
In both the cases, we prove that $(a,b)\in A_m$ in the following three steps.\\
\textbf{Step 1:} $(a,b)\in\mathbb{Z}^2$. It is clear from the definitions of $a$ and $b$.\\ 
\textbf{Step 2:} $a^2+216b^2=jm$. It is easy to check that in both the cases
\begin{align*}
a^2+216b^2=\frac{(8x^2_1+27y^2_1)(8u^2+27v^2)}{p^2}=jm.
\end{align*}\\
\textbf{Step 3:} $\gcd(a,b)=1$. First, we write $u$ and $v$ in terms of $a$ and $b$. From \eqref{definition_a_b_1}, we get
\begin{align}\label{u_vcase1}
u=\frac{x_1a+27y_1b}{j},\ \ \ \  v=\frac{8x_1b-y_1a}{j},
\end{align}
and from \eqref{definition_a_b_2}, we get
\begin{align}\label{u_vcase2}
u=\frac{x_1a-27y_1b}{j},\ \ \ \  v=\frac{8x_1b+y_1a}{j}.
\end{align}
Now, if $j=1$ then from \eqref{u_vcase1} and \eqref{u_vcase2}, we have that $\gcd(a,b)\mid\gcd(u,v)$. Since $\gcd(u,v)=1$, it follows that $\gcd(a,b)=1$. If $j=4,8$, then $y_1$ is even and therefore, $x_1$ is odd, since $\gcd(x_1,y_1)=1$. Also, $v$ is odd as $pm$ is odd. Thus, in both the cases, $a$ is even and $b$ is odd. Then, $\gcd(a,b)$ is an odd number and from \eqref{u_vcase1} and \eqref{u_vcase2}, it follows that $\gcd(a,b)\mid\gcd(u,v)$. Since $\gcd(u,v)=1$, we have $\gcd(a,b)=1$.  
\par Next, we show that $f$ is surjective and two-to-one. Let $(a,b)\in A_m$. Define $u,~v$ as defined in \eqref{u_vcase1} and $\tilde{u},~\tilde{v}$ as defined in \eqref{u_vcase2}. We prove that $(u,v), (\tilde{u},\tilde{v})\in U_{p,m}$ in the following three steps.\\
\textbf{Step 1:} $(u,v),(\tilde{u},\tilde{v})\in\mathbb{Z}^2$. It is clear in the case when $j=1$. If $j=4$,  we have $a^2+216b^2=4m$ and $8x^2_1+27y^2_1=4p$. Clearly, $2\mid a,y_1$ and therefore, $b,~x_1$ are odd. Also, $4\nmid a,y_1$. Thus, $u,\tilde{u}\in\mathbb{Z}$. Since $4\mid y_1a$, we have $v,\tilde{v}\in\mathbb{Z}$.
\par If $j=8$, then $8x^2_1+27y^2_1=8p$ and from Lemma \ref{lemma_j=8} it follows that $p\equiv1,7\pmod{24}$.\\
\underline{Case 1:} $p\equiv1\pmod{24}$. By Lemma \ref{lemma_j=8}, we have $val_2(y_1)\geq3$. As $pm\equiv11\pmod{24}$, we have $m\equiv11\pmod{24}$. Since $a^2+216b^2=8m$, from Lemma 2.6 of \cite{Merca}, it follows that $val_2(a)\geq3$. Therefore, $u,\tilde{u}\in\mathbb{Z}$. Since $8\mid y_1a$, we have $v,\tilde{v}\in\mathbb{Z}$.\\
\underline{Case 2:} $p\equiv7\pmod{24}$. By Lemma \ref{lemma_j=8}, we have $val_2(y_1)=2$. As $pm\equiv11\pmod{24}$, we have $m\equiv5\pmod{24}$. Since $a^2+216b^2=8m$, from Lemma 2.6 of \cite{Merca}, it follows that $val_2(a)=2$ and therefore, $b, x_1$ are odd.  Thus, $u,\tilde{u}\in\mathbb{Z}$. Since $8\mid y_1a$, we have $v,\tilde{v}\in\mathbb{Z}$.\\
\textbf{Step 2:} $8u^2+27v^2=pm$ and $8\tilde{u}^2+27\tilde{v}^2=pm$. It is easy to check that
\begin{align*}
8u^2+27v^2=\frac{(8x^2_1+27y^2_1)(a^2+216b^2)}{j^2}=pm.
\end{align*} 
Similarly, $8\tilde{u}^2+27\tilde{v}^2=pm$.\\
\textbf{Step 3:} $\gcd(u,v)=\gcd(\tilde{u},\tilde{v})=1$. If we express $a$ and $b$ in terms of $u$ and $v$, we get \eqref{definition_a_b_1}. We have $8u^2+27v^2=pm$. Clearly, $p\nmid u$ and $p\nmid v$. Therefore, $p\nmid\gcd(u,v)$ and from \eqref{definition_a_b_1} it follows that $\gcd(u,v)\mid\gcd(a,b)$. Since $\gcd(a,b)=1$, we have $\gcd(u,v)=1$. Similarly, $\gcd(\tilde{u},\tilde{v})=1$.
\par From \eqref{definition_a_b_1} and \eqref{definition_a_b_2}, we get
\begin{align}\label{a_1}
8x_1u-27y_1v&\equiv0\pmod{p};\\
8x_1u+27y_1v&\equiv0\pmod{p}.\label{a_2}
\end{align}
This shows that $f$ is surjective and $f^{-1}\{ (a,b)\} =\left\lbrace(u,v),~(\tilde{u},\tilde{v})\right\rbrace$.
\par Finally, we prove that $(u,v)\neq(\tilde{u},\tilde{v})$. Suppose $(u,v)=(\tilde{u},\tilde{v})$, then both $(u,v)$ and $(\tilde{u},\tilde{v})$ satisfy \eqref{a_1} and \eqref{a_2}. This implies that $x_1u\equiv0\pmod{p}$. If $p\mid x_1$, then from $8x^2_1+27y^2_1=jp$, we get that $p\mid y_1$ which is a contradiction, as $\gcd(x_1,y_1)=1$. Similarly, $p\nmid u$. Hence, $(u,v)\neq(\tilde{u},\tilde{v})$. This shows that $f$ is two-to-one.
\end{proof}
\begin{corollary}\label{main_lemma_cor}
If $p,~m,~j,$ and $U_{p,m}$ are as defined in Lemma \ref{U_p,m_A_m}, then for $m>1$ $|U_{p,m}|\equiv0\pmod{8}$ and $|U_{p,1}|=4$.
\end{corollary}
\begin{proof}
If $m=1$, then from Lemma 2.4 of \cite{Merca}, $a^2+216b^2=jm$ has primitive solutions if $j=1$. Therefore, $A_m=\left\lbrace(\pm0,1)\right\rbrace$, i.e., $|A_m|=2$ and $|U_{p,m}|=4$. If $m>1$, then for every $(a,b)\in A_m$, $a$ and $b$ are nonzero. Thus, every member $(a,b)$ of set $A_m$ ensures four distinct elements $(\pm a,\pm b)$ in $A_m$. Hence, $|A_m|\equiv0\pmod{4}$, and by Lemma \ref{U_p,m_A_m}, we get $|U_{p,m}|\equiv0\pmod{8}$. 
\end{proof}
Having all the required lemmas proved, we are now ready to prove Theorem \ref{thm1}.
\begin{proof}[Proof of Theorem \ref{thm1}]
Define 
\begin{align*}
M_1(u)&:=\left|\left\lbrace(x,y)\in\mathbb{Z}^2:8x^2+3y^2=u\right\rbrace\right|;\\
M_2(u)&:=\left|\left\lbrace(x,y)\in\mathbb{Z}^2:8x^2+27y^2=u\right\rbrace\right|;\\
N_1(u)&:=\left|\left\lbrace(x,y)\in M_1(u):\gcd(x,y)=1\right\rbrace\right|;\\
N_2(u)&:=\left|\left\lbrace(x,y)\in M_2(u):\gcd(x,y)=1\right\rbrace\right|.
\end{align*}
Note that there are exactly four positive, reduced, primitive quadratic forms of discriminant $-96$, namely, $x^2 + 24y^2;~ 3x^2+8y^2;~5x^2+2xy+5y^2;~4x^2+4xy+7y^2$. Let $N(-96,u)$ be the number of primitive representations of any integer $u$ by positive, reduced, primitive quadratic forms of discriminant $-96$. Then, by  \cite[Lemma 3.25]{Cox}
\begin{align*}
N(-96,u)=2\prod_{p\mid u}\left(1+\left(\frac{-6}{p}\right)_J\right),
\end{align*}
where the product is over all the prime divisors of $u$ and $\left(\frac{-6}{p}\right)_J$ is the Jacobi symbol. A positive integer $u\equiv  1, 5, 7, 11 \pmod{24}$ can be represented by one and only one of the four quadratic forms above,  as proved in \cite[pg. 84]{Dickson}. Therefore, if $u\equiv 11\pmod{24}$, then
\begin{align}\label{N_1_formula}
N_1(u)=N(-96,u).
\end{align}   
\par Observe that if $u$ is square free, then $M_i(u)=N_i(u)$ for $i\in\left\lbrace1,2\right\rbrace$. If $u=wv^2$, for some positive integers $w$ and $v$ with $w$ square free. Then, for $i\in\left\lbrace1,2\right\rbrace$
\begin{align*}
M_i(u)=\sum_{d\mid v}N_i(wd^2),
\end{align*}
where the sum is over all the positive divisors of $u$. From \cite[pg. 84]{Dickson}, if $u\equiv11\pmod{24}$, then we have
\begin{align}\label{M_1_formula}
M_1(u)=2\sum_{d\mid u}\left(\frac{-6}{d}\right)_J,
\end{align}
where the sum is over all positive divisors of $u$ and $\left(\frac{-6}{d}\right)_J$ is the Jacobi symbol.
\par By Lemma \ref{lemma_1}, for $t=p^2n+\beta p-11\cdot24^{-1}_p$, we have $24t+11=pm$, for some positive integer $m$ such that $p\nmid m$. If $(x_0,y_0)\in\mathbb{Z}^2$ is a solution for $8x^2+3y^2=24t+11$, since $24t+11=pm$ is not a perfect square, we have $x_0,y_0\neq0$. Then, $M_i(pm)\equiv0\pmod{4}$. Similarly, $N_i(pm)\equiv0\pmod{4}$.
\par We need to prove that $b_{21}(4t+1)\equiv0\pmod{2}$. Since $pm\equiv11\pmod{24}$ and $24t+11=pm$ is not perfect square, we have 
\begin{align*}
b_{21}(4t+1)&\equiv a(t)\\
&\equiv\frac{1}{4}\left(M_1(24t+11)-M_2(24t+11)\right)\\
&\equiv\frac{1}{4}\left(M_1(pm)-M_2(pm)\right)\pmod{2}. 
\end{align*}
Therefore, it is enough to show that
\begin{align*}
M_1(pm)-M_2(pm)\equiv0\pmod{8}.
\end{align*}
\par If $m=1$, we have $p\equiv11\pmod{24}$. Then, from \eqref{M_1_formula}, we get $M_1(p)=4$. From Corollary \ref{main_lemma_cor}, it follows that $M_2(p)=N_2(p)=|U_{p,1}|=4$. Hence, $M_1(p)-M_2(p)\equiv0\pmod{8}$.
\par If $m>1$, then from \eqref{N_1_formula}, we have $N_1(pm)\equiv0\pmod{8}$.\\
\underline{Case 1:} If $m$ is not a perfect square, then $pm=pwv^2$ for some positive integers $w,v$ with $w$ square free. Therefore, for $i\in\left\lbrace1,2\right\rbrace$,
\begin{align*}
	M_i(pm)=\sum_{d\mid v}N_i(pwd^2).
\end{align*}
Since $pm\equiv11\pmod{24}$, therefore, $\gcd(m,6)=\gcd(v,6)=1$, and also $\gcd(d,6)=1$ for all $d\mid v$. Thus, $v^2\equiv1\pmod{24}$ and $d^2\equiv1\pmod{24}$ for all $d\mid v$. Then for all $d\mid v$, we have $pwd^2\equiv pw\equiv pwv^2\equiv pm\equiv11\pmod{24}$. By Corollary \ref{main_lemma_cor}, $N_2(pwd^2)\equiv0\pmod{8}$ for all $d\mid v$, and thus, $M_2(pm)\equiv0\pmod{8}$. Also, from \eqref{N_1_formula}, $N_1(pwd^2)\equiv0\pmod{8}$ for all $d\mid v$ and therefore, $M_1(pm)\equiv0\pmod{8}$. Hence, $M_1(pm)-M_2(pm)\equiv0\pmod{8}$.\\
\underline{Case 2:} If $m$ is a perfect square, since $\gcd(m,6)=1$, it follows that $m\equiv1\pmod{24}$, and therefore, $p\equiv11\pmod{24}$. Then, from  \eqref{M_1_formula}, $N_1(p)=M_1(p)=4$ and from Corollary \eqref{main_lemma_cor}, $N_2(p)=|U_{p,1}|=4$. Let $m=v^2$, for some positive integer $v$. For $i\in\left\lbrace1,2\right\rbrace$, we have 
\begin{align*}
	M_i(pm)=N_i(p)+\sum_{d\mid v, d>1}N_i(pd^2).
\end{align*}
From Corollary \ref{main_lemma_cor}, it follows that $N_2(pd^2)\equiv0\pmod{8}$ for all $d\mid v$, $d>1$. Also, by \eqref{N_1_formula}, $N_1(pd^2)\equiv0\pmod{8}$ for all $d\mid v$, $d>1$. Therefore, for $i\in\left\lbrace1,2\right\rbrace$, 
\begin{align*}
M_i(pm)\equiv N_i(p)\equiv4\pmod{8}.
\end{align*}
Hence, $M_1(pm)-M_2(pm)\equiv0\pmod{8}$. This completes the proof.
\end{proof}
Next, we prove that Theorem \ref{thm1} holds for infinitely many primes. To prove this claim, we show that the Dirichlet density of the set $\mathcal{Q}$ is positive. Moreover, we show that Theorem \ref{thm1} holds for infinitely many primes congruent to $k$ modulo $24$ for each $k\in\{1, 5, 7, 11\}$. First, we introduce some definition and notation.
\begin{definition}
Let $P$ denote the set of all prime numbers. Then the Dirichlet density of set $S\subset P$, denoted by $\delta(S)$, is defined as
\begin{align*}
\delta(S):=\lim_{s\rightarrow1^{+}}\frac{\sum_{p\in S}p^{-s}}{\sum_{p\in P}p^{-s}}.
\end{align*}
\end{definition}
For $k\in \{1,5,7,11\}$, let $\mathcal{Q}_k:=\{p\in \mathcal{Q}:p\equiv k\pmod{24}\}$.
\begin{proposition}\label{Prop1}
We have $\delta(\mathcal{Q})=\frac{1}{6}$, and $\delta(\mathcal{Q}_k)=\frac{1}{24}$ for each $k\in\left\lbrace1,5,7,11\right\rbrace$.
\end{proposition}
\begin{remark}
In \cite[Proposition 2.9]{Merca}, Ballantine, Merca and Radu proved that the Dirichlet density of $\mathcal{P}$ is equal to $1/6$ and $\delta(\mathcal{P}_k)=1/24$ for $k\in\left\lbrace1,5,7,11\right\rbrace$, where $\mathcal{P}_k:=\{p\in \mathcal{P}:p\equiv k\pmod{24}\}$. Recall that $\mathcal{P}$ denotes the set of primes $p$ such that the Diophantine equation $x^2+24\cdot 9y^2=jp$ has primitive solutions for some $j\in \{1, 4, 8\}$. A proof of Proposition \ref{Prop1} can be given as done in \cite[Proposition 2.9]{Merca}. However, we prove Proposition \ref{Prop1} by showing that the sets $\mathcal{Q}$ and $\mathcal{P}$ are the same, and hence $\mathcal{P}_k=\mathcal{Q}_k$ for $k=1, 5, 7, 11$. 
\end{remark}
\begin{proof}[Proof of Proposition \ref{Prop1}]
We consider four cases depending on $p$ modulo $24$ and in each case, we prove that the sets $\mathcal{Q}$ and $\mathcal{P}$ contain the same set of primes.\\
\underline{Case I} $p\equiv1\pmod{24}$. Let $p\in\mathcal{Q}$. Then by Lemma \ref{lemma_j=8}, $8x^2+27y^2=8p$ has primitive solutions, say $(x_0,y_0)$, such that $8\mid y_0$. Since $\gcd(x_0,y_0)=1$, we have $\gcd(x_0,\frac{y_0}{8})=1$ and therefore, $(x_0,\frac{y_0}{8})$ is a primitive solution for $x^2+216y^2=p$. Thus, $p\in\mathcal{P}$.\\
Conversely, let $p\in\mathcal{P}$. Then, from \cite[Lemma 2.4]{Merca}, $x^2+216y^2=p$ has primitive solutions, say $(x_1,y_1)$. Clearly, $2\nmid x_1$ and therefore, $\gcd(x_1,8y_1)=1$. Thus, $(x_1,8y_1)$ is a primitive solution for $8x^2+27y^2=8p$. This implies that $p\in\mathcal{Q}$.\\
\underline{Case II} $p\equiv5\pmod{24}$. Let $p\in\mathcal{Q}$. Then by Lemma \ref{lemma_j=4}, $8x^2+27y^2=4p$ has primitive solutions, say $(x_0,y_0)$. Clearly, $2\mid y_0$ and $4\nmid y_0$. This implies that $\gcd(4x_0,\frac{y_0}{2})=1$ and thus, $(4x_0,\frac{y_0}{2})$ is a primitive solution for $x^2+216y^2=8p$. Therefore, $p\in\mathcal{P}$.\\
Conversely, let $p\in\mathcal{P}$. Then, from \cite[Lemma 2.6]{Merca}, $x^2+216y^2=8p$ has primitive solutions, say $(x_1,y_1)$, such that $4\mid x_1$ but $8\nmid x_1$. Since $\gcd(x_1,y_1)=1$, we have $\gcd(\frac{x_1}{4},2y_1)=1$ and therefore, $(\frac{x_1}{4},2y_1)$ is a primitive solution for $8x^2+27y^2=4p$. Thus, $p\in\mathcal{Q}$.\\
\underline{Case III} $p\equiv7\pmod{24}$. Let $p\in\mathcal{Q}$. As seen in Case I, from Lemma \ref{lemma_j=8}, $8x^2+27y^2=8p$ has primitive solutions, say $(x_0,y_0)$, such that $4\mid y_0$ but $8\nmid y_0$. Then, $(2x_0,\frac{y_0}{4})$ is a primitive solution for $x^2+216y^2=4p$ and therefore, $p\in\mathcal{P}$.\\
Conversely, let $p\in\mathcal{P}$. Then, from \cite[Lemma 2.5]{Merca}, $x^2+216y^2=4p$ has primitive solutions, say $(x_1,y_1)$. Clearly, $2\mid x_1$ and $4\nmid x_1$, and therefore, $(\frac{x_1}{2},4y_1)$ is a primitive solution for $8x^2+27y^2=8p$. Thus, $p\in\mathcal{Q}$.\\
\underline{Case IV} $p\equiv11\pmod{24}$. Let $p\in\mathcal{Q}$. Then by Lemma \ref{lemma_j=1}, $8x^2+27y^2=p$ has primitive solutions, say $(x_0,y_0)$. Here, $2\nmid y_0$ and therefore, $(8x_0,y_0)$ is a primitive solution for $x^2+216y^2=8p$. Hence, $p\in\mathcal{P}$.\\
Conversely, let $p\in\mathcal{P}$. Then, from \cite[Lemma 2.6]{Merca}, $x^2+216y^2=8p$ has primitive solutions, say $(x_1,y_1)$, such that $8\mid x_1$. Therefore, $(\frac{x_1}{8},y_1)$ is a primitive solution for $8x^2+27y^2=p$. Thus, $p\in\mathcal{Q}$.
\par This completes the proof of the proposition.
\end{proof}
In \cite[Theorem 1.4]{SB1}, the first and the third author proved Theorem \ref{thm1} for $p=29$ using the approach developed in \cite{radu1, radu2}. We now give another proof of Theorem \ref{thm1} for the primes $p=59,79$ using that approach. Throughout this section, $\Gamma$ denotes the full modular group $\text{SL}_2(\mathbb{Z})$ and define 
  \begin{align*}
  \Gamma_{\infty} & :=\left\{
\begin{bmatrix}
1  &  n \\
0  &  1      
\end{bmatrix}: n\in \mathbb{Z}  \right\}.
  \end{align*}
We recall that the index of $\Gamma_{0}(N)$ in $\Gamma$ is
\begin{align*}
 [\Gamma : \Gamma_0(N)] = N\prod_{p|N}(1+p^{-1}),
\end{align*}
where $p$ denotes a prime.
\par 
For a positive integer $M$, let $R(M)$ be the set of integer sequences $r=(r_\delta)_{\delta|M}$ indexed by the positive divisors of $M$. 
If $r \in R(M)$ and $1=\delta_1<\delta_2< \cdots <\delta_k=M$ 
are the positive divisors of $M$, we write $r=(r_{\delta_1}, \ldots, r_{\delta_k})$. Define $c_r(n)$ by 
\begin{align}
\sum_{n=0}^{\infty}c_r(n)q^n:=\prod_{\delta|M}(q^{\delta};q^{\delta})^{r_{\delta}}_{\infty}=\prod_{\delta|M}\prod_{n=1}^{\infty}(1-q^{n \delta})^{r_{\delta}}.
\end{align}
The approach to prove congruences for $c_r(n)$ developed by Radu \cite{radu1, radu2} reduces the number of coefficients that one must check as compared with the classical method which uses Sturm's bound alone.
\par 
Let $m$ be a positive integer. For any integer $s$, let $[s]_m$ denote the residue class of $s$ in $\mathbb{Z}_m:= \mathbb{Z}/ {m\mathbb{Z}}$. 
Let $\mathbb{Z}_m^{*}$ be the set of all invertible elements in $\mathbb{Z}_m$. Let $\mathbb{S}_m\subseteq\mathbb{Z}_m$  be the set of all squares in $\mathbb{Z}_m^{*}$. For $t\in\{0, 1, \ldots, m-1\}$
and $r \in R(M)$, we define a subset $P_{m,r}(t)\subseteq\{0, 1, \ldots, m-1\}$ by
\begin{align*}
P_{m,r}(t):=\left\{t': \exists [s]_{24m}\in \mathbb{S}_{24m} ~ \text{such} ~ \text{that} ~ t'\equiv ts+\frac{s-1}{24}\sum_{\delta|M}\delta r_\delta \pmod{m} \right\}.
\end{align*}
\begin{definition}
	Suppose $m, M$, and $N$ are positive integers, $r=(r_{\delta})\in R(M)$, and $t\in \{0, 1, \ldots, m-1\}$. Let $k=k(m):=\gcd(m^2-1,24)$ and write  
	\begin{align*}
	\prod_{\delta|M}\delta^{|r_{\delta}|}=2^s\cdot j,
	\end{align*}
	where $s$ and $j$  are nonnegative integers with $j$ odd. The set $\Delta^{*}$ consists of all tuples $(m, M, N, (r_{\delta}), t)$ satisfying these conditions and all of the following.
	\begin{enumerate}
		\item Each prime divisor of $m$ is also a divisor of $N$.
		\item $\delta|M$ implies $\delta|mN$ for every $\delta\geq1$ such that $r_{\delta} \neq 0$.
		\item $kN\sum_{\delta|M}r_{\delta} mN/\delta \equiv 0 \pmod{24}$.
		\item $kN\sum_{\delta|M}r_{\delta} \equiv 0 \pmod{8}$.  
		\item  $\frac{24m}{\gcd{(-24kt-k{\sum_{{\delta}|M}}{\delta r_{\delta}}},24m)}$ divides $N$.
		\item If $2|m$, then either $4|kN$ and $8|sN$ or $2|s$ and $8|(1-j)N$.
	\end{enumerate}
\end{definition}
Let $m, M, N$ be positive integers. For $\gamma=
\begin{bmatrix}
	a  &  b \\
	c  &  d     
\end{bmatrix} \in \Gamma$, $r\in R(M)$ and $r'\in R(N)$, set 
	\begin{align*}
	p_{m,r}(\gamma):=\min_{\lambda\in\{0, 1, \ldots, m-1\}}\frac{1}{24}\sum_{\delta|M}r_{\delta}\frac{\gcd^2(\delta a+ \delta k\lambda c, mc)}{\delta m}
	\end{align*}
and 
\begin{align*}
	p_{r'}^{*}(\gamma):=\frac{1}{24}\sum_{\delta|N}r'_{\delta}\frac{\gcd^2(\delta, c)}{\delta}.
\end{align*}
\begin{lemma}\label{lem1}\cite[Lemma 4.5]{radu1} Let $u$ be a positive integer, $(m, M, N, r=(r_{\delta}), t)\in\Delta^{*}$ and $r'=(r'_{\delta})\in R(N)$. 
Let $\{\gamma_1,\gamma_2, \ldots, \gamma_n\}\subseteq \Gamma$ be a complete set of representatives of the double cosets of $\Gamma_{0}(N) \backslash \Gamma/ \Gamma_\infty$. 
Assume that $p_{m,r}(\gamma_i)+p_{r'}^{*}(\gamma_i) \geq 0$ for all $1 \leq i \leq n$. Let $t_{min}=\min_{t' \in P_{m,r}(t)} t'$ and
\begin{align*}
\nu:= \frac{1}{24}\left\{ \left( \sum_{\delta|M}r_{\delta}+\sum_{\delta|N}r'_{\delta}\right)[\Gamma:\Gamma_{0}(N)] -\sum_{\delta|N} \delta r'_{\delta}\right\}-\frac{1}{24m}\sum_{\delta|M}\delta r_{\delta} 
	- \frac{ t_{min}}{m}.
\end{align*}	
If the congruence $c_r(mn+t')\equiv0\pmod u$ holds for all $t' \in P_{m,r}(t)$ and $0\leq n\leq \lfloor\nu\rfloor$, then it holds for all $t'\in P_{m,r}(t)$ and $n\geq0$.
\end{lemma}
	To apply Lemma \ref{lem1}, we utilize the following result, which gives us a complete set of representatives of the double coset in $\Gamma_{0}(N) \backslash \Gamma/ \Gamma_\infty$. 
\begin{lemma}\label{lem2}\cite[Lemma 4.3]{wang} If $N$ or $\frac{1}{2}N$ is a square-free integer, then
		\begin{align*}
		\bigcup_{\delta|N}\Gamma_0(N)\begin{bmatrix}
		1  &  0 \\
		\delta  &  1      
		\end{bmatrix}\Gamma_ {\infty}=\Gamma.
		\end{align*}
\end{lemma}
Here, we prove the theorem for $p=59$ and omit the details of the proof for the case when $p=79$.
\par By \eqref{thm1.1}, we have
\begin{align}
	\sum_{n=0}^{\infty}b_{21}(4n+1)q^n\equiv \frac{f_3^4}{f_1}\pmod 2.
\end{align}
Let $(m,M,N,r,t)=(3481,3,177,(-1,4),297)$. We verify that $(m,M,N,r,t) \in \Delta^{*}$ and $ P_{m,r}(t)=\{ 2,61,120,297,356,415,474,592,651,710,887,1064,1182,1300,1359,\\ 1418,1536,1713,1949,2067,2185,2244,2362,2421,2657,2952,3011,3365,3424\}$.
By Lemma \ref{lem2}, we know that $\left\{\begin{bmatrix}
	1  &  0 \\
	\delta  &  1      
\end{bmatrix}:\delta|177 \right\}$ forms a complete set of double coset representatives of $\Gamma_{0}(N) \backslash \Gamma/ \Gamma_\infty$.
Let $\gamma_{\delta}=\begin{bmatrix}
	1  &  0 \\
	\delta  &  1      
\end{bmatrix}$. Let $r'=(0,0,0,0)\in R(177)$ and we use $Sage$ to verify that
$p_{m,r}(\gamma_{\delta})+p_{r'}^{*}(\gamma_{\delta}) \geq 0$ for each $\delta | N$. We compute that the upper bound in Lemma \ref{lem1} is $\lfloor\nu\rfloor=29$. Using $Sage$, we verify that $b_{21}(4(3481n+t')+1)\equiv0\pmod{2}$ for all $t' \in P_{m,r}(t)$ and for $n\leq 29$. By Lemma \ref{lem1}, we conclude that $b_{21}(4(3481n+t')+1)\equiv0\pmod{2}$ for all $t' \in P_{m,r}(t)$ and for all $n\geq 0$. Next, we take $(m,M,N,r,t)=(3481,3,177,(-1,4),533)$. It is easy to verify that $(m,M,N,r,t) \in \Delta^{*}$ and $P_{m,r}(t)=\{179,238,533,769,828,946,1005,1123,1241,\\ 1477, 1654,1772,1831,1890,2008,2126,2303,2480,2539,2598,2716,2775,2834, 2893,\\ 3070,3129,3188,3247,3306\}$. Following similar steps as shown before, we find that $b_{21}(4(3481n+t')+1)\equiv0\pmod{2}$ for all $t' \in P_{m,r}(t)$ and for all $n\geq 0$.
\par Theorem \ref{thm1}, for $p=59$ says that
\begin{align*}
b_{21}(4(3481n+59\beta+297)+1)\equiv 0 \pmod{2},
\end{align*}
for $0\leq\beta<59$, $\beta\neq22$. Notice that the sets $P_{m,r}(t)$ include 58 numbers, excluding $1595=59\cdot22+297$. This completes the proof when $p=59$.
\par Similarly, for $p=79$, we take $(m,M,N,r,t)=(6241,3,237,(-1,4),253),(6241,3,\\ 237,(-1,4),332)\in\Delta^{*}$ and in both the cases, we get $\lfloor\nu\rfloor=39$ . Using $Sage$, we verify that $b_{21}(4(6241n+t')+1)\equiv0\pmod{2}$ for all $t' \in P_{m,r}(t)$ and for $n\leq39$. Therefore, by Lemma \ref{lem1} we conclude that $b_{21}(4(6241n+t')+1)\equiv0\pmod{2}$ for all $t' \in P_{m,r}(t)$ and for all $n\geq 0$. This completes the proof when $p=79$.
\section{Proof of theorem \ref{Newmann_b21}}
\begin{proof}
We define
\begin{align}\label{thm5.1}
\sum_{n=0}^{\infty}c(n)q^n:=\frac{f^4_3}{f_1}.
\end{align}
Then by \eqref{thm1.1} and \eqref{thm5.1}, for all $n\geq0$, we get
\begin{align}\label{relation_in_b_c}
b_{21}(4n+1)\equiv c(n)\pmod{2}.
\end{align}
Now, Theorem 3 of \cite{Newman} yields
\begin{align}\label{newmann_main}
c\left(p^2n+11\left( \frac{p^2-1}{24}\right) \right)-\gamma(n)c(n)+c\left(\frac{n-11\left( \frac{p^2-1}{24}\right) }{p^2}\right) \equiv0\pmod{2}.
\end{align}
In \eqref{newmann_main}
\begin{align}\label{alpha_value}
\gamma(n)\equiv pd(p)+\left(\frac{n-11\left( \frac{p^2-1}{24}\right) }{p}\right)_L\pmod{2},
\end{align}
where $d(p)$ is a function of $p$ and $\left(\frac{.}{p}\right)_L$ denotes the Legendre symbol. Note that $c(0)=1$ and $c(\ell)=0$ for all $\ell<0$. Put $n=0$ in \eqref{newmann_main}, to get
\begin{align}\label{2.3}
c\left(11\left(\frac{p^2-1}{24}\right) \right)\equiv\gamma(0)\pmod{2}. 
\end{align}
By \eqref{alpha_value}, we get
\begin{align}\label{2.3.1}
\gamma(0)\equiv pd(p)+1\pmod{2}.
\end{align}
Employing \eqref{2.3.1} in \eqref{2.3}, we arrive at
\begin{align}\label{2.3.2}
pd(p)\equiv c\left(11\left(\frac{p^2-1}{24}\right)\right)+1\pmod{2}. 
\end{align}
Then \eqref{alpha_value} and \eqref{2.3.2} yield
\begin{align}\label{alpha_new_value}
\gamma(n)\equiv c\left(11\left(\frac{p^2-1}{24}\right)\right)+1+\left(\frac{n-11\left( \frac{p^2-1}{24}\right) }{p}\right)_L\pmod{2}.
\end{align}
Combining \eqref{newmann_main} and \eqref{alpha_new_value} gives
\begin{align}\label{newmann_main_1}
c\left(p^2n+11\left( \frac{p^2-1}{24}\right) \right)\equiv&\left(c\left(11\left(\frac{p^2-1}{24}\right)\right)+1+\left(\frac{n-11\left( \frac{p^2-1}{24}\right) }{p}\right)_L\right)c(n)\nonumber\\
&+c\left(\frac{n-11\left( \frac{p^2-1}{24}\right) }{p^2}\right)\pmod{2}.
\end{align}
Replacing $n$ by $pn+11\left( \frac{p^2-1}{24}\right)$ in \eqref{newmann_main_1}, we deduce that
\begin{align}\label{newmann_main_2}
c\left(p^3n+11\left( \frac{p^4-1}{24}\right) \right)\equiv&\left(c\left(11\left(\frac{p^2-1}{24}\right)\right)+1\right)c\left( pn+11\left( \frac{p^2-1}{24}\right)\right)\nonumber\\ 
&+c\left(\frac{n}{p}\right)\pmod{2}.
\end{align}
\underline{Case $(i)$:}~$b_{21}\left(\frac{11p^2-5}{6}\right)\equiv1\pmod{2}$. In this case, by \eqref{relation_in_b_c},  $c\left(11\left( \frac{p^2-1}{24}\right) \right)\equiv1\pmod{2}$. Then from \eqref{newmann_main_2}, we have
\begin{align}\label{newmann_main_3}
c\left(p^3n+11\left( \frac{p^4-1}{24}\right) \right)\equiv c\left(\frac{n}{p}\right)\pmod{2}.
\end{align}
Next, replace $n$ by $pn+\beta$ in \eqref{newmann_main_3} with $1\leq\beta<p$:
\begin{align}\label{newmann_main_4}
c\left(p^3\left( pn+\beta\right) +11\left( \frac{p^4-1}{24}\right) \right)\equiv0\pmod{2}.
\end{align}
Substituting $n$ by $pn$ in \eqref{newmann_main_3}, we obtain
\begin{align}\label{newmann_main_4.1}
c\left(p^4n+11\left( \frac{p^4-1}{24}\right) \right)\equiv c\left(n\right)\pmod{2}.
\end{align}
Using \eqref{newmann_main_4.1} repeatedly, we obtain that, for $n, k\geq0$,
\begin{align}\label{newmann_main_5}
c\left(p^{4k}n+11\left( \frac{p^{4k}-1}{24}\right) \right)\equiv c\left(n\right)\pmod{2}.
\end{align}
Finally, replacing $n$ by $p^3(pn+\beta)+11\left( \frac{p^4-1}{24}\right)$ in \eqref{newmann_main_5}, we arrive at
\begin{align}\label{newmann_main_6}
c\left(p^{4k+4}n+p^{4k+3}\beta +11\left( \frac{p^{4k+4}-1}{24}\right) \right)\equiv0\pmod{2}.
\end{align}
Combining \eqref{relation_in_b_c} and \eqref{newmann_main_6}, we deduce \eqref{Yao_1}.
\par Putting $n=0$ in \eqref{newmann_main_5}, we get
\begin{align*}
	c\left(11\left( \frac{p^{4k}-1}{24}\right) \right)\equiv1\pmod{2},
\end{align*}
which, when combined with \eqref{relation_in_b_c}, yields \eqref{Yao_1a}.\\
\underline{Case $(ii)$:}~$b_{21}\left(\frac{11p^2-5}{6}\right)\equiv0\pmod{2}$. In this case, by \eqref{relation_in_b_c},  $c\left(11\left( \frac{p^2-1}{24}\right) \right)\equiv0\pmod{2}$. Then, substituting $n$ by $np$ in \eqref{newmann_main_2}, we get
\begin{align}\label{newmann_main_7}
	c\left(p^4n+11\left( \frac{p^4-1}{24}\right) \right)\equiv c\left( p^2n+11\left( \frac{p^2-1}{24}\right)\right)+c(n)\pmod{2}.
\end{align}
Substituting $n$ by $p^2n+11\left( \frac{p^2-1}{24}\right)$ in \eqref{newmann_main_7}, we arrive at the following re-occurrence relation:
\begin{align}\label{newmann_main_7.1}
	c\left(p^6n+11\left( \frac{p^6-1}{24}\right) \right)\equiv c(n)\pmod{2}.
\end{align}
Iterating re-occurrence relation \eqref{newmann_main_7.1}, for $n, k\geq0$, we get
\begin{align}\label{newmann_main_8}
	c\left(p^{6k}n+11\left( \frac{p^{6k}-1}{24}\right) \right)\equiv c(n)\pmod{2}.
\end{align}
Notice that, in this case, if we take all those $n\geq0$ for which $p\nmid(24n+11)$, then \eqref{newmann_main_1} takes the following form:
\begin{align*}
c\left(p^2n+11\left(\frac{p^2-1}{24}\right) \right)\equiv0\pmod{2}. 
\end{align*}
Replacing $n$ by $p^2n+11\left(\frac{p^2-1}{24}\right)$ in \eqref{newmann_main_8}, we have
\begin{align}\label{newmann_main_9}
	c\left(p^{6k+2}n+11\left( \frac{p^{6k+2}-1}{24}\right) \right)\equiv0\pmod{2}.
\end{align}
Finally, combining \eqref{relation_in_b_c} and \eqref{newmann_main_9}, we get \eqref{Yao_2}. Putting $n=0$ in \eqref{newmann_main_8}, we get
\begin{align*}
	c\left(11\left( \frac{p^{6k}-1}{24}\right) \right)\equiv1\pmod{2},
\end{align*}
which, when combined with \eqref{relation_in_b_c}, yields \eqref{Yao_2a}.
\end{proof}
\section{Concluding Remarks}
In this section, we propose the following conjecture for $b_{21}(n)$.
\begin{conj}\label{conj}
If $p\in\mathcal{Q}$, then $b_{21}\left(\frac{11p^2-5}{6}\right)\equiv 1\pmod{2}$.
\end{conj}
A proof of Conjecture \ref{conj} will confirm that the congruences obtained in Theorem \ref{thm1} are different from those obtained in Theorem \ref{Newmann_b21}. We add the following discussion on Conjecture \ref{conj}.
\par Let $p\in \mathcal{Q}$. By taking $k=\frac{11p^2-11}{24}$ in \eqref{relation_b21_and_a}, we get
\begin{align*}
b_{21}\left(\frac{11p^2-5}{6}\right)\equiv a\left(\frac{11p^2-11}{24}\right)\pmod{2},
\end{align*}
where
\begin{align*}
a\left(\frac{11p^2-11}{24}\right)=\left|\left\lbrace(x,y)\in\mathbb{N}^2:\ 8x^2+3y^2=11p^2,~3\nmid y\right\rbrace \right|.
\end{align*}
Since $11p^2\equiv11\pmod{24}$, \eqref{M_1_formula} yields
\begin{align*}
	\left|\left\lbrace(x,y)\in\mathbb{Z}^2:\ 8x^2+3y^2=11p^2\right\rbrace \right|=M_1(11p^2)=12.
\end{align*}
Out of $12$ solutions of $8x^2+3y^2=11p^2$, four non-primitive solutions are $(\pm p,\pm p)$. We conjecture that for $p\in \mathcal{Q}$, if the remaining $8$ solutions are $(\pm a, \pm b)$, $(\pm u, \pm v)$, then $3\nmid b,v$. That is, $a\left(\frac{11p^2-11}{24}\right)=3$ and thus, $b_{21}\left(\frac{11p^2-5}{6}\right)$ is odd. For example, when $p=29$, $(\pm 16,\pm 49),~(\pm 29,\pm 29)$, and $(\pm 34,\pm 1)$ are all solutions of $8x^2+3y^2=11\cdot {29}^2$ and hence, $a\left(\frac{11\cdot{29}^2-11}{24}\right)=3$. Therefore, to prove Conjecture \ref{conj}, it is enough to prove that for any $p\in\mathcal{Q}$, there is no integral solutions to $8x^2+27y^2=11p^2$. 


\begin{thebibliography}{999}
\bibitem{Andrews}
G. E. Andrews, {\it The Theory of Partitions}, Cambridge Mathematical Library, Cambridge University Press, Cambridge, 1998. Reprint of the 1976 original.

\bibitem{Merca}
C. Ballantine, M. Merca, and S. Radu, {\it Parity of $3$-regular partition numbers and Diophantine equations}, arXiv:2212.09810 [math.NT].

\bibitem{baruahdas}
N. D. Baruah and K. Das, {\it Parity results for $7$-regular and $23$-regular partitions}, Int. J. Number Theory 11 (2015), 2221--2238.

\bibitem{calkin}
N. Calkin, N. Drake, K. James, S. Law, P. Lee, D. Penniston, and J. Radder, {\it Divisibility properties of the $5$-regular and $13$-regular partition functions}, Integers (2008), A60, 10 pp.

\bibitem{Carlson}
R. Carlson and J. J. Webb, {\it Infinite families of infinite families of congruences for $k$-regular partitions}, Ramanujan J. 33 (2014), 329--337. 

  
\bibitem{Cox}
D. Cox, {\it Primes of the form $x^2+ny^2$. Fermat, class field theory, and complex multiplication}, Second edition. Pure and Appl. Math., John Wiley $\&$ Sons, Inc., Hoboken, NJ, 2013.
 
\bibitem{cui_Gu_2013}
S. -P. Cui and N. S. S. Gu, {\it Arithmetic properties of $\ell$-regular partitions}, Adv. in Appl. Math. 51 (2013), 507--523.

\bibitem{cui_Gu_2015}
S. -P. Cui and N. S. S. Gu, {\it Congruences for $9$-regular partitions modulo $3$}, Ramanujan J. 38 (2015), 503--512.

\bibitem{Dandurand}
B. Dandurand and D. Penniston, {\it $\ell$-Divisibility of $\ell$-regular partition functions}, Ramanujan J. 19 (2009), 63--70.

\bibitem{Dickson}
L. Dickson, {\it Introduction to the Theory of Numbers}, University of Chicago Press, Chicago, 1929.

\bibitem{Furcy}
D. Furcy and D. Penniston, {\it Congruences for $\ell$-regular partition functions modulo $3$}, Ramanujan J. 27 (2012), 101--108.

\bibitem{gordon1997}
B. Gordon and K. Ono, {\it Divisibility of certain partition functions by powers of primes}, Ramanujan J. 1 (1997), 25--34.

\bibitem{hirschhorn_sellers}
M. D. Hirschhorn and J. A. Sellers, {\it Elementary proofs of parity results for $5$-regular partitions}, Bull. Aust. Math. Soc. 81 (2010), 58--63.

\bibitem{Keith2014}
W. J. Keith, {\it Congruences for $9$-regular partitions modulo $3$}, Ramanujan J. 35 (2014), 157--164.

\bibitem{Keith2021}
W. J. Keith and F. Zanello, {\it Parity of the coefficients of certain eta-quotients}, J. Number Theory 235 (2022), 275--304.

\bibitem{Lovejoy}
J. Lovejoy and D. Penniston, {\it $3$-regular partitions and a modular $K3$ surface}, Contemp. Math. 291 (2001) 177--182.



\bibitem{Newman} 
M. Newman, {\it Modular forms whose coefficients possess multiplicative properties, II}, Ann. Math. 75 (1962), 242--250.

\bibitem{Penniston_2002}
D. Penniston, {\it The $p^a$-regular partition function modulo $p^j$}, J. Number Theory 94 (2002), 320--325.

\bibitem{Penniston_2008}
D. Penniston, {\it Arithmetic of $\ell$-regular partition functions}, Int. J. Number Theory 4 (2008), 295--302.

\bibitem{radu1}
S. Radu, {\it An algorithmic approach to Ramanujan's congruences}, Ramanujan J.  20 (2)  (2009), 295--302.

\bibitem{radu2}
S. Radu and J. A. Sellers, {\it Congruence properties modulo $5$ and $7$ for the pod  function}, Int. J. Number Theory 7 (8) (2011), 2249--2259.

\bibitem{SB1}
A. Singh and R. Barman, {\it Divisibility of certain $\ell$-regular partitions by $2$}, Ramanujan J. 59 (2022), 813--829.

\bibitem{SB2}
A. Singh and R. Barman, {\it Proofs of some conjectures of Keith and Zanello on $t$-regular partition}, Pacific J. Mathematics, to appear.



\bibitem{wang}
L. Wang, {\it Arithmetic properties of $(k, \ell)$-regular bipartitions}, Bull. Aust. Math. Soc.  95 (2017), 353--364.

\bibitem{Webb}
J. J. Webb, {\it Arithmetic of the $13$-regular partition function modulo $3$}, Ramanujan J. 25 (2011), 49--56.

\bibitem{Xia_2015}
E. X. W. Xia, {\it Congruences for some l-regular partitions modulo l}, J. Number Theory 152 (2015), 105--117.

\bibitem{xia_yao}
E. X. W. Xia and O. X. M. Yao, {\it Parity results for $9$-regular partitions}, Ramanujan J. 34 (2014), 109--117.

\bibitem{Yao2022}
O. X. M. Yao, {\it New parity results for $3$-regular partitions}, Quaestiones Mathematicae (2022) DOI: 10.2989/16073606.2022.2033872. 


\bibitem{zhao_jin_yao}
T. Y. Zhao, J. Jin, and O. X. M. Yao, {\it Parity results for $11$-, $13$- and $17$-regular partitions}, Coll. Mathematicum 151 (2018), 97--109.
\end{thebibliography}
\end{document}